\documentclass{amsart}

\usepackage[table,xcdraw]{xcolor}
\usepackage{amssymb}
\usepackage{amsmath}
\usepackage{amsfonts}
\usepackage{amsxtra}
\usepackage{amsthm}
\usepackage{tikz}
\usepackage{tikz-cd}
\usepackage{float}
\usepackage{booktabs}
\usepackage{color}
\usepackage{xcolor}
\usepackage{graphicx}
\usepackage{enumitem}
\usepackage{venndiagram}
\usetikzlibrary{positioning,shapes,fit,arrows}
\usepackage{multirow}
\usepackage{multicol}
\usepackage{url}
\usepackage{hyperref}
\usepackage{listings}

\newtheorem{theorem}{Theorem}[section]
\newtheorem{proposition}[theorem]{Proposition}
\newtheorem{corollary}[theorem]{Corollary}
\newtheorem{lemma}[theorem]{Lemma}
\newtheorem{exmp}[theorem]{Example}
\newtheorem{cexmp}[theorem]{Counterexample}
\newtheorem{definition}[theorem]{Definition}

\newcommand{\abf}{\mathbf{a}}

\newcommand{\Id}{\mathrm{Id}}

\def\f{\mathbb{F}}

\def\n{\mathbb{N}}
\def\z{\mathbb{Z}}
\def\q{\mathbb{Q}}
\def\r{\mathbb{R}}
\def\cc{\mathbb{C}}

\def\inv{^{-1}}

\definecolor{codegreen}{rgb}{0,0.6,0}
\definecolor{codegray}{rgb}{0.5,0.5,0.5}
\definecolor{codepurple}{rgb}{0.58,0,0.82}
\definecolor{backcolour}{rgb}{0.95,0.95,0.92}

\lstdefinestyle{mystyle}{
    backgroundcolor=\color{backcolour},   
    commentstyle=\color{codegreen},
    keywordstyle=\color{magenta},
    numberstyle=\tiny\color{codegray},
    stringstyle=\color{codepurple},
    basicstyle=\ttfamily\footnotesize,
    breakatwhitespace=false,         
    breaklines=true,                 
    captionpos=b,                    
    keepspaces=true,                 
    numbers=left,                    
    numbersep=5pt,                  
    showspaces=false,                
    showstringspaces=false,
    showtabs=false,                  
    tabsize=2
}

\begin{document}

\title[]{Flasque Meadows}

\author[]{João Dias}
\author[]{Bruno Dinis}

\address[João Dias]{Departamento de Matemática, Universidade de Évora}
\email{joao.miguel.dias@uevora.pt}

\address[Bruno Dinis]{Departamento de Matemática, Universidade de Évora}
\email{bruno.dinis@uevora.pt}



\begin{abstract}
In analogy with flasque sheaves, we introduce the notion of flasque meadow as a common meadow where the transition maps are all surjective. We study some properties of flasque meadows and illustrate them with many examples and counterexamples.

\end{abstract}

\maketitle

\section{Introduction}

The notion of common meadow has gained some attention lately due to recently discovered connections with ring theory and with nonstandard analysis \cite{Dinis_Bottazzi,Dias_Dinis(23),Dias_Dinis(24),Dias_Dinis(Art),DDM(ta)} (see also \cite{Bergstra_Tucker_2024, bergstra2024ringscommondivisioncommon} for other recent developments).  Originally introduced by Bergstra and Tucker in \cite{Bergstra2006}, \emph{meadows} are algebraic structures with two operations, addition and multiplication, where both operations have inverses which are total. This implies in particular that one is allowed to divide by zero. The original idea was to view meadows as abstract data types given by equational axiomatizations \cite{Bergstra2008,BP(20),10.1145/1219092.1219095,bergstra2020arithmetical}. The point is that one is then able to obtain simple term rewriting systems which are easier to automate in formal reasoning  \cite{bergstra2020arithmetical,bergstra2023axioms}. However, and as mentioned above, meadows are also algebraic structures whose study has interest on its own. For our purposes, two different kinds of meadows are particularly relevant: \emph{common meadows}, introduced by Bergstra and Ponse in \cite{Bergstra2015} and \emph{pre-meadows with $\abf$}, introduced by the authors in \cite{Dias_Dinis(23)}.  In these classes of meadows the inverse of zero is an element, denoted $\abf$, which is absorbent for both operations. The interesting feature of these two classes is that they can be decomposed as disjoint unions of rings \cite{Dias_Dinis(23)}, which allows to adapt many of the results of ring theory onto them. Pre-meadows with $\abf$ differ from common meadows as the former do not require the existence of an inverse function. This turns out to be crucial in trying to enumerate finite common meadows \cite{Dias_Dinis(24)}. The two structures are nevertheless related since $P$, a pre-meadow with $\abf$ such that $0\cdot P$ is finite and the partial order in $0\cdot P$ is a total order, is actually a common meadow.

In this paper we introduce a new class of meadows, called \emph{flasque meadows}, inspired by a recently discovered connection between meadows and sheaves \cite{DDM(ta)}. The term \emph{flasque} originates from Sheaf Theory, where a sheaf is said to be flasque if the restriction maps are surjective \cite[p.~67]{Hartshorne}. Flasque sheaves have the nice property of preserving some exact sequences.  We study some basic properties of flasque meadows and illustrate them with many examples and counterexamples. 

\section{Preliminaries}
In this section we recall some definitions and results on common meadows.

\begin{definition}
        A \emph{pre-meadow} is a structure $(P,+,-,\cdot)$ satisfying the following equations
    \begin{multicols}{2}
\begin{enumerate}
\item[$(P_1)$] $(x+y)+z=x+(y+z) $
\item[$(P_2)$] $x+y=y+x $ 
\item[$(P_3)$]  $x+0=x$ 
\item[$(P_4)$] $x+ (-x)=0 \cdot x$
\item[$(P_5)$] $(x \cdot y) \cdot z=x \cdot (y \cdot z)$ 
\item[$(P_6)$]  $x \cdot y=y \cdot x $
\item[$(P_7)$] $1 \cdot x=x$
\item[$(P_8)$] $x \cdot (y+z)= x \cdot y + x \cdot z$
\item[$(P_9)$] $-(-x)=x$
\item[$(P_{10})$] $0\cdot(x+y)=0\cdot x \cdot y$
\end{enumerate}
\end{multicols}

 Let $P$ be a pre-meadow and let $P_z:=\{x\in M\mid 0\cdot x=z\}$. We say that $P$ is a \emph{pre-meadow with $\abf$} if there exists a unique $z\in 0\cdot P$ such that $|P_z|=1$ (denoted by $\abf$) and $x+\abf=\abf$, for all $x\in P$. 
    A \emph{common meadow} is a pre-meadow with $\abf$ equipped with an inverse function $(\cdot)\inv$ satisfying
    \begin{multicols}{2}
\begin{enumerate}
\item[$(M_1)$] $x \cdot x^{-1}=1 + 0 \cdot x^{-1}$
\item[$(M_2)$]$(x \cdot y)^{-1} = x^{-1} \cdot y^{-1}$
\item[$(M_3)$] $(1 + 0 \cdot x)^{-1} = 1 + 0 \cdot x $
\item[$(M_4)$] $ 0^{-1}=\abf$
\end{enumerate}
\end{multicols}
\end{definition}

\begin{exmp}
\begin{enumerate}
    \item The set $P=\{0,1,x,\abf\}$ with the operations
\begin{multicols}{2}
     \begin{table}[H]
\begin{tabular}{|c|clll}
\hline
$+$ & \multicolumn{1}{c|}{$0$} & \multicolumn{1}{l|}{$1$} & \multicolumn{1}{l|}{$x$} & \multicolumn{1}{l|}{$\abf$} \\ \hline
$0$ & $0$                      & $1$                      & $x$                      & $\abf$                      \\ \cline{1-1}
$1$ & $1$                      & $0$                      & $x$                      & $\abf$                      \\ \cline{1-1}
$x$ & $x$                      & $x$                      & $x$                      & $\abf$                                         \\ \cline{1-1}
  $\abf$ &               $\abf$          &                $\abf$         &               $\abf$          &              $\abf$                             \\ \cline{1-1}
\end{tabular}
\end{table}

    \begin{table}[H]
\begin{tabular}{|c|clll}
\hline
$\cdot$ & \multicolumn{1}{c|}{$0$} & \multicolumn{1}{l|}{$1$} & \multicolumn{1}{l|}{$x$} & \multicolumn{1}{l|}{$\abf$} \\ \hline
$0$ & $0$                      & $0$                      & $x$                      & $\abf$                      \\ \cline{1-1}
$1$ & $0$                      & $1$                      & $x$                      & $\abf$                      \\ \cline{1-1}
$x$ & $x$                      & $x$                      & $x$                      & $\abf$                                         \\ \cline{1-1}
  $\abf$ &               $\abf$          &                $\abf$         &               $\abf$          &              $\abf$                             \\ \cline{1-1}
\end{tabular}
\end{table}
\end{multicols}

    is a pre-meadow but not a pre-meadow with $\abf$ since, the fact that $|P_x|=1=|P_\abf|$ would entail $x=\abf$.

    \item The set $M=\{0,1,x,y,\abf\}$ with the operations
    \begin{multicols}{2}
     \begin{table}[H]
\begin{tabular}{|c|cllll}
\hline
$+$                         & \multicolumn{1}{c|}{$0$} & \multicolumn{1}{l|}{$1$} & \multicolumn{1}{l|}{$x$} & \multicolumn{1}{l|}{$y$} & \multicolumn{1}{l|}{$\abf$} \\ \hline
$0$       & $0$& $1$                      & $x$                      & $y$                         & $\abf$      \\ \cline{1-1}
$1$       & $1$ & $0$                      & $y$                      &     $x$                     & $\abf$      \\ \cline{1-1}
$x$       & $x$    & $y$                      & $x$                      &     $y$                     &  $\abf$           \\ \cline{1-1}
\multicolumn{1}{|l|}{$y$} & $y$  &     $x$    & $y$                       &             $x$             &       $\abf$      \\ \cline{1-1}
  $\abf$      &        $\abf$                   &        $\abf$                 &         $\abf$                &      $\abf$ &$\abf$                                   \\ \cline{1-1}
\end{tabular}
\end{table}
     \begin{table}[H]
\begin{tabular}{|c|cllll}
\hline
$\cdot$                     & \multicolumn{1}{c|}{$0$} & \multicolumn{1}{l|}{$1$} & \multicolumn{1}{l|}{$x$} & \multicolumn{1}{l|}{$y$} & \multicolumn{1}{l|}{$\abf$} \\ \hline
$0$       & $0$& $0$                      & $x$                      & $x$                         & $\abf$      \\ \cline{1-1}
$1$       & $0$ & $1$                      & $x$                      &     $y$                     & $\abf$      \\ \cline{1-1}
$x$       & $x$    & $x$                      & $x$                      &     $x$                     &  $\abf$           \\ \cline{1-1}
\multicolumn{1}{|l|}{$y$} & $x$  &     $y$    & $x$                       &             $y$             &       $\abf$      \\ \cline{1-1}
  $\abf$      &        $\abf$                   &        $\abf$                 &         $\abf$                &      $\abf$ &$\abf$                                   \\ \cline{1-1}
\end{tabular}
\end{table}
\end{multicols}

    is a pre-meadow with $\abf$. Indeed, $M_0=\{0,1\}$, $M_x=\{x,y\}$ and $M_\abf=\{\abf\}$. Moreover, $M$ becomes a common meadow by making $1\inv =1$, $y\inv =y$, and $0\inv=x\inv = \abf\inv=\abf$.
     \end{enumerate}
    \end{exmp}

    A \emph{directed lattice of rings} is a lattice where the vertices are labelled by unital commutative rings and the edges are labelled by ring homomorphisms. 

      In \cite{Dias_Dinis(23)} a connection between meadows and directed lattices was unveiled (see Theorem~\ref{T:DirectedLattice} below) which allows to represent meadows in a diagram form. We illustrate this with the following example. 

        \begin{exmp}
    Consider the lattice $L$ over the set $\{1,2,3,4\}$ (on the left)
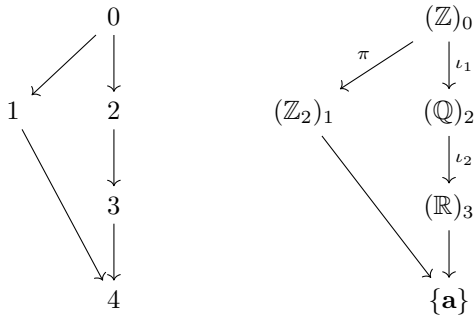
\[\begin{tikzcd}
	& 0 &&& {(\z)_0} \\
	1 & 2 && {(\z_2)_1} & {(\q)_2} \\
	& 3 &&& {(\r)_3} \\
	& 4 &&& {\{\abf\}}
	\arrow[from=1-2, to=2-1]
	\arrow[from=1-2, to=2-2]
	\arrow["\pi"', from=1-5, to=2-4]
	\arrow["{\iota_1}", from=1-5, to=2-5]
	\arrow[from=2-1, to=4-2]
	\arrow[from=2-2, to=3-2]
	\arrow[from=2-4, to=4-5]
	\arrow["{\iota_2}", from=2-5, to=3-5]
	\arrow[from=3-2, to=4-2]
	\arrow[from=3-5, to=4-5]
\end{tikzcd}\]
    
    A directed lattice of rings $\Gamma = \bigsqcup \Gamma_i$ over $L$ is a labeling of the vertices by unital commutative rings $R_i$ and of the arrows by unital ring homomorphisms.

Here (on the right), $\pi$ is the projection homomorphism to the quotient $\z/(2)\simeq \z_2$, and $\iota_1,\iota_2$ are the inclusion homomorphisms (note that $\iota_2\circ \iota_1$ is the inclusion map of $\z$ into $\r$). The unlabelled arrows correspond to the unique map that exists to the trivial ring, that we denote by $\{\abf\}$, which we usually omit since it is unique. Throughout the paper we will also drop the lattice indices.
\end{exmp}

    We now recall the main result from \cite{Dias_Dinis(23)}, linking rings and common meadows. The set of invertible elements of the ring $R$ is denoted $R^{\times}$.
\begin{theorem}[{\cite[Theorem~3.3]{Dias_Dinis(23)}}]\label{T:DirectedLattice}
        Let $L$ be a lattice and $\Gamma$ a directed lattice of rings over  $L$ such that, for all $i$ in a set of indices $I$ and all $x\in \Gamma_i$ the set:
        $$J_x:=\{j\in I\mid f_{j,i}(x)\in \Gamma_j^{\times}\}$$
        has a greatest element.
          Then, there exists a common meadow $M=\bigsqcup_{i\in L}\Gamma_i$ with the operations 
          \begin{itemize}
                \item $x+_My=f_{i\wedge j,i}(x)+_{i\wedge j}f_{i\wedge j,j}(y)$, where $+_{i\wedge j}$ is the sum in $\Gamma_{i\wedge j}$;
                \item $x \cdot_M y=f_{i\wedge j,i}(x)\cdot_{i\wedge j}f_{i\wedge j,j}(y)$, where $\cdot_{i\wedge j}$ is the product in $\Gamma_{i\wedge j}$.
            \end{itemize}
             such that the lattice $0\cdot M$ is equivalent to $L$.
\end{theorem}

The following corollary is useful to show immediately that certain directed lattices are common meadows, since it entails that if all the vertices of the associated lattice of a pre-meadow with $\abf$ are fields, then it must be a common meadow.
        \begin{corollary}[\cite{Dias_Dinis(Art)}]\label{P:VerticesAreFields}
            Every $P$, pre-meadow with $\abf$ such that $P_{0\cdot z}$ is a field, for all $0\cdot z\in 0\cdot P$, is a common meadow.
        \end{corollary}

\begin{exmp}
    Recall that for $p$ prime, there exists a unique field $\f_{p^n}$ with exactly $p^n$ elements. Also, the field $\f_{p^n}$ is an algebraic extension of $\f_{p^m}$ whenever $m|n$. Consider the following directed lattice
    
   
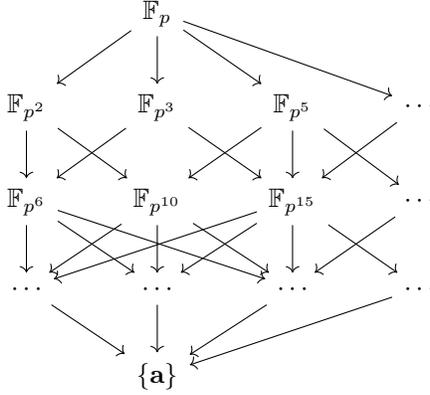
\[\begin{tikzcd}
	& {\f_p} \\
	{\f_{p^2}} & {\f_{p^3}} & {\f_{p^5}} & \cdots \\
	{\f_{p^6}} & {\f_{p^{10}}} & {\f_{p^{15}}} & \cdots \\
	\cdots & \cdots & \cdots & \cdots \\
	& {\{\abf\}}
	\arrow[from=1-2, to=2-1]
	\arrow[from=1-2, to=2-2]
	\arrow[from=1-2, to=2-3]
	\arrow[from=1-2, to=2-4]
	\arrow[from=2-1, to=3-1]
	\arrow[from=2-1, to=3-2]
	\arrow[from=2-2, to=3-1]
	\arrow[from=2-2, to=3-3]
	\arrow[from=2-3, to=3-2]
	\arrow[from=2-3, to=3-3]
        \arrow[from=2-3, to=3-4]
        \arrow[from=2-4, to=3-3]
	\arrow[from=3-1, to=4-1]
	\arrow[from=3-1, to=4-2]
	\arrow[from=3-1, to=4-3]
	\arrow[from=3-2, to=4-1]
	\arrow[from=3-2, to=4-2]
	\arrow[from=3-2, to=4-3]
	\arrow[from=3-3, to=4-1]
	\arrow[from=3-3, to=4-2]
	\arrow[from=3-3, to=4-3]
 	\arrow[from=3-3, to=4-4]
  	\arrow[from=3-4, to=4-3]
	\arrow[from=4-1, to=5-2]
	\arrow[from=4-2, to=5-2]
	\arrow[from=4-3, to=5-2]
 	\arrow[from=4-4, to=5-2]
\end{tikzcd}\]
We have that  $M=\bigsqcup_{n\in \n} \f_{p^n}$ is a common meadow, by Corollary \ref{P:VerticesAreFields}.
\end{exmp}

\section{Flasque Meadows}
    As shown in \cite[Proposition~3.4]{Dias_Dinis(Art)}, there exists a class of common meadows $M$ which is completely determined by $M_0$. We will see (in Proposition~\ref{P:Determined}) that in the class of pre-meadows with $\abf$ a similar phenomenon occurs. We recall that, given a pre-meadow $P$,  the transition maps of $P$ are ring homomorphisms $f_{0\cdot w,0\cdot z}:P_{0\cdot w}\to P_{0\cdot z}$. If the maps are surjective, then $P_{0\cdot w}/\ker(f_{0\cdot w,0\cdot z})$ is isomorphic to $P_{0\cdot z}$. The case where the maps are surjective is of great interest since, in that case, $P_0$ defines the entire meadow. 

\begin{definition}
    Let $P$ be a pre-meadow with $\abf$. We say that $P$ is \emph{flasque} if for all $0\cdot z,0\cdot w\in 0\cdot P$ the transition maps $f_{0\cdot w,0\cdot z}:P_{0\cdot w}\to P_{0\cdot z}$ are surjective.
\end{definition}

    The first examples of flasque pre-meadows are the pre-meadows with $\abf$ constructed from a ring, i.e.\ of the form 
    \begin{equation}\label{E:M(R)}
    M(R):=\bigsqcup_{I\trianglelefteq R}R/I,    
    \end{equation}
     for some unital commutative ring $R$ (see \cite{Dias_Dinis(Art)} for more on this class of meadows).
    
    \begin{exmp}
        The directed lattice
            
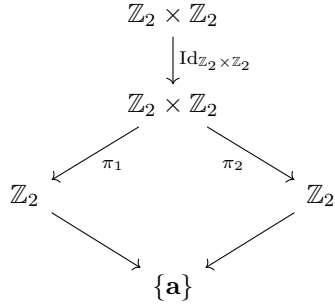
\[\begin{tikzcd}
	& {\z_2\times\z_2} \\
	& {\z_2\times\z_2} \\
	{\z_2} && {\z_2} \\
	& {\{\abf\}}
	\arrow["{\Id_{\z_2\times\z_2}}", from=1-2, to=2-2]
	\arrow["{\pi_1}", from=2-2, to=3-1]
	\arrow["{\pi_2}"', from=2-2, to=3-3]
	\arrow[from=3-1, to=4-2]
	\arrow[from=3-3, to=4-2]
\end{tikzcd}\]

        where the maps $\pi_1$ and $\pi_2$ are the projection maps on the first and second coordinates, respectively, defines a flasque common meadow. 
    \end{exmp}

\begin{exmp}
    Consider the ring $\z_2\times \z_3\times \z_5$. The directed lattice of rings associated with $M(R)$ is
    {\scriptsize
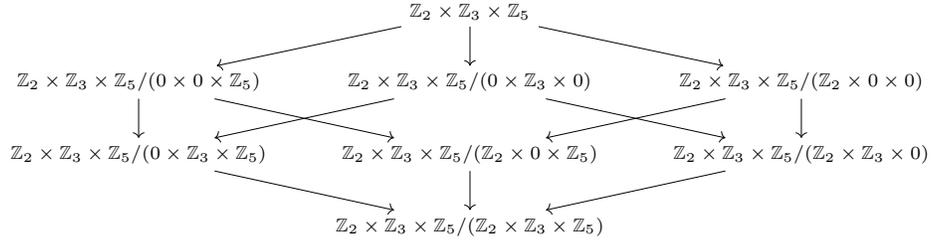
    \[\begin{tikzcd}
	& {\z_2\times\z_3\times\z_5} \\
	{\z_2\times\z_3\times\z_5/(0\times0\times \z_5)} & {\z_2\times\z_3\times\z_5/(0\times \z_3\times 0)} & {\z_2\times\z_3\times\z_5/(\z_2\times 0\times 0)} \\
	{\z_2\times\z_3\times\z_5/(0\times \z_3\times \z_5)} & {\z_2\times\z_3\times\z_5/(\z_2\times 0\times \z_5)} & {\z_2\times\z_3\times\z_5/(\z_2\times \z_3\times 0)} \\
	& {\z_2\times\z_3\times\z_5/(\z_2\times \z_3\times \z_5)}
	\arrow[from=1-2, to=2-1]
	\arrow[from=1-2, to=2-2]
	\arrow[from=1-2, to=2-3]
	\arrow[from=2-1, to=3-1]
	\arrow[from=2-1, to=3-2]
	\arrow[from=2-2, to=3-1]
	\arrow[from=2-2, to=3-3]
	\arrow[from=2-3, to=3-2]
	\arrow[from=2-3, to=3-3]
	\arrow[from=3-1, to=4-2]
	\arrow[from=3-2, to=4-2]
	\arrow[from=3-3, to=4-2]
\end{tikzcd}\]
}

    which is equivalent to
    
    
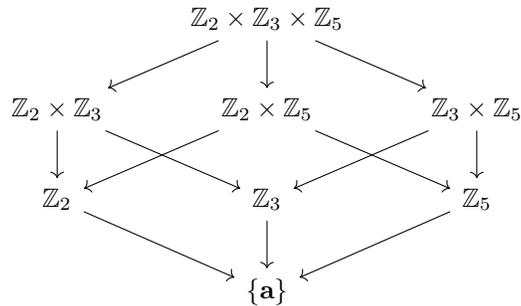
\[\begin{tikzcd}
	& {\z_2\times\z_3\times\z_5} \\
	{\z_2\times\z_3} & {\z_2\times\z_5} & {\z_3\times\z_5} \\
	{\z_2} & {\z_3} & {\z_5} \\
	& {\{\abf\}}
	\arrow[from=1-2, to=2-1]
	\arrow[from=1-2, to=2-2]
	\arrow[from=1-2, to=2-3]
	\arrow[from=2-1, to=3-1]
	\arrow[from=2-1, to=3-2]
	\arrow[from=2-2, to=3-1]
	\arrow[from=2-2, to=3-3]
	\arrow[from=2-3, to=3-2]
	\arrow[from=2-3, to=3-3]
	\arrow[from=3-1, to=4-2]
	\arrow[from=3-2, to=4-2]
	\arrow[from=3-3, to=4-2]
\end{tikzcd}\]
    Where the transition maps are the natural quotient homomorphisms.
\end{exmp}

    \begin{cexmp}\label{CE:Notallareflasque}
    \begin{enumerate}
        \item Not all flasque pre-meadow with $\abf$ are common meadows. Indeed, the pre-meadow with $\abf$ associated with the following directed lattice is not a common meadow
        
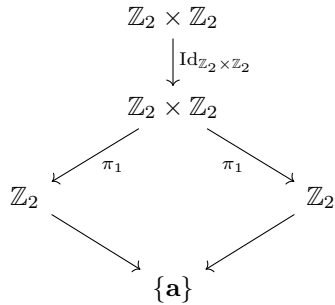
\[\begin{tikzcd}
	& {\z_2\times\z_2} \\
	& {\z_2\times\z_2} \\
	{\z_2} && {\z_2} \\
	& {\{\abf\}}
	\arrow["{\Id_{\z_2\times\z_2}}", from=1-2, to=2-2]
	\arrow["{\pi_1}", from=2-2, to=3-1]
	\arrow["{\pi_1}"', from=2-2, to=3-3]
	\arrow[from=3-1, to=4-2]
	\arrow[from=3-3, to=4-2]
\end{tikzcd}\].
        \item Not all common meadows are flasque. Indeed the directed lattice
       
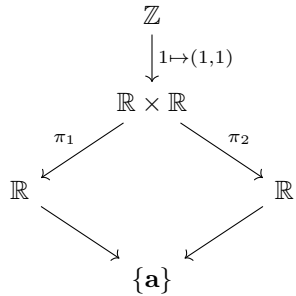
\[\begin{tikzcd}
	& \z \\
	& {\r\times \r} \\
	\r && \r \\
	& {\{\abf\}}
	\arrow["{1\mapsto (1,1)}", from=1-2, to=2-2]
	\arrow["{\pi_1}"', from=2-2, to=3-1]
	\arrow["{\pi_2}", from=2-2, to=3-3]
	\arrow[from=3-1, to=4-2]
	\arrow[from=3-3, to=4-2]
\end{tikzcd}\]
    is associated with a common meadow. However this meadow is not flasque since the map $1\mapsto (1,1)$ fails to be surjective.

    \end{enumerate}
             \end{cexmp}    

\begin{proposition}\label{P:M(R)isflasque}
    Let $R$ be a commutative ring. Then $M(R)$ is a flasque pre-meadow with $\abf$.
\end{proposition}

\begin{proof}
Since the transition maps are projections from $R/I$ to $R/J$ with $I\subseteq J$ then they are surjective, and so $M(R)$ is flasque.
\end{proof}

We now give an alternative condition for a pre-meadow with $\abf$ to be flasque.

\begin{proposition}\label{L:Surjectiveequiva}
    Let $P$ be a pre-meadow with $\abf$ then the following are equivalent:
    \begin{enumerate}
        \item for all $0\cdot z,0\cdot w\in 0\cdot P$ the transition map $f_{0\cdot w,0\cdot z}:P_{0\cdot w}\to P_{0\cdot z}$ is surjective.
        \item for all $0\cdot z\in 0\cdot P$ the transition map $f_{0,0\cdot z}:P_0\to P_{0\cdot z}$ is surjective.
    \end{enumerate}
\end{proposition}
\begin{proof}
    The implication $(1)\Rightarrow (2)$ is obvious.
    
    Assume that for all $0\cdot z\in 0\cdot P$ the transition map $f_{0,0\cdot z}:P_0\to P_{0\cdot z}$ is surjective. Let  $0\cdot z,0\cdot w\in 0\cdot P$. We have that $f_{0\cdot z}= f_{0\cdot w,0\cdot z}\circ f_{0\cdot w}$ is surjective, so $f_{0\cdot w,0\cdot z}$ is also surjective.
\end{proof}

If $P$ is a flasque pre-meadow with $\abf$, then $P$ is determined by $P_0$ and the lattice structure of $0\cdot P$. The following proposition reflects this fact.

\begin{proposition}\label{P:Determined}
    Let $P$ be a flasque pre-meadow with $\abf$. Then $P=P_0+0\cdot P$.
\end{proposition}
\begin{proof}
    By Proposition \ref{L:Surjectiveequiva},  for all $x\in P$ there exists some $y\in P_0$ such that  $x=y+0\cdot x$. Then  $P\subseteq P_0+0\cdot P$, and so $P=P_0+0\cdot P$ as we wanted.
\end{proof}

As a consequence of Proposition~\ref{P:Determined}, every $P$, a flasque pre-meadow with $\abf$, decomposes into the sum of a ring $P_0$ with a lattice  $0\cdot P$. Moreover, if $P$ is a flasque pre-meadow with $\abf$ and $P_0$ is a field, then $P$ must be a common meadow, by Proposition \ref{L:Surjectiveequiva}. 

\begin{theorem}\label{T:Decomposition}
    Let $M$ be a common meadow. Then  $M'=M_0\times (0\cdot M\setminus \{\abf\} )\sqcup\{\abf\}$, with the operations
    $$
    \begin{cases}
        (x,0\cdot z)+(y,0\cdot w)= (x+y,0\cdot z\cdot w)\\
        (x,0\cdot z)\cdot(y,0\cdot w)= (x\cdot y,0\cdot z\cdot w),
    \end{cases}
    $$
    where $(x,0\cdot z),(y,0\cdot w)\in M'$, is a common meadow. In particular, if $M$ is flasque and such that $M_0$ is a field, then $M$ is isomorphic to $M'$.
\end{theorem}
\begin{proof}
    Since $(M_0,+,\cdot)$ is a ring and $0\cdot M$ is a semigroup, both operations are associative on $M'$. Let $(x_i,0\cdot z_i)\in M'$, with $i=1,2,3$. Then
    \begin{align*}
     (x_1,z_1)\cdot ( (x_2,z_2)+ (x_3,z_3) ) &= (x_1\cdot(x_2+x_3),z_1\cdot z_2\cdot z_3) =\\  
     (x_1\cdot x_2+x_1\cdot x_3,z_1\cdot z_2\cdot z_1\cdot z_3)&=(x_1,z_1)\cdot  (x_2,z_2)+  (x_1,z_1)\cdot (x_3,z_3)
    \end{align*}
    and so distributivity holds on $M'$. The elements $(0,0)$ and $(1,0)$ are the identity of the sum and the product, respectively. For $(x,z)\in M'$ we have
    $$
    (x,z)\inv =\begin{cases}
        (x\inv,z),\text{ if }x\in M_0^\times\\
        \abf,\text{ otherwise}
    \end{cases}
    $$
    The remaining common meadow axioms are easy to very. 

    If $M$ is a flasque common meadow such that $M_0$ is a field, then  for all $0\cdot z\in 0\cdot M$ $M_{0\cdot z}\simeq M_0$. Since all surjective homomorphisms between fields are isomorphisms, the transition maps are all isomorphisms. Let $\Phi:M'\to M$ be the function defined by $\Phi(x,0\cdot z)=f_{0,0\cdot z}(x)=x+0\cdot z$, where $f_{0,0\cdot z}$ are the transition maps of $M$. Since all transition maps are isomorphisms,  $\Phi$ must be bijective. Let us see that it is in fact an homomorphism of common meadows. Let $(x,0\cdot z),(y,0\cdot w)\in M'$. Then 
    \begin{align*}
        \Phi(x,0\cdot z)+\Phi(y,0\cdot w)&= x+0\cdot z + y +0\cdot w=\\
        x+y + 0\cdot z\cdot w &= \Phi(x+y,0\cdot z \cdot w). 
    \end{align*}
    Similarly, one shows that $\Phi(x,0\cdot z)\cdot \Phi(y,0\cdot w)=\Phi(x\cdot y,0\cdot z \cdot w)$. Also $\Phi(1,0)=1$. Hence $\Phi$ is an isomorphism of common meadows, and so $M'$ is isomorphic to $M$.
\end{proof}
\begin{exmp}
    Consider the common meadow $M$ (which  is not flasque) associated with the directed lattice on the left. Then $M'$ can be represented by the directed lattice on the right, where the non-trivial arrows correspond to the identity function.
    
    \begin{enumerate}
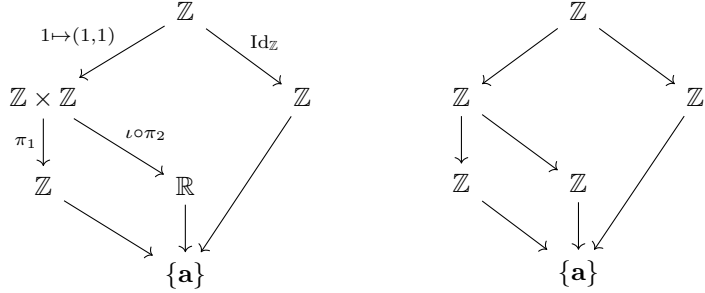
     \begin{multicols}{2}
        \item[]       
        \begin{tikzcd}
	& \z \\
	{\z\times\z} && \z \\
	\z & \r \\
	& {\{\abf\}} \\
	\arrow["{1\mapsto (1,1)}"', from=1-2, to=2-1]
	\arrow["{\Id_\z}", from=1-2, to=2-3]
	\arrow["{\pi_1}"', from=2-1, to=3-1]
	\arrow["{\iota\circ\pi_2}", from=2-1, to=3-2]
	\arrow[from=2-3, to=4-2]
	\arrow[from=3-1, to=4-2]
	\arrow[from=3-2, to=4-2]
\end{tikzcd}
\item[]
\begin{tikzcd}
	& \z \\
	\z && \z \\
	\z & \z \\
	& {\{\abf\}}
	\arrow[from=1-2, to=2-1]
	\arrow[from=1-2, to=2-3]
	\arrow[from=2-1, to=3-1]
	\arrow[from=2-1, to=3-2]
	\arrow[from=2-3, to=4-2]
	\arrow[from=3-1, to=4-2]
	\arrow[from=3-2, to=4-2]
\end{tikzcd}
  \end{multicols}
  \end{enumerate}
\end{exmp}

The proof of Theorem \ref{T:Decomposition} also shows that, given a common meadow $M$, there exists an homomorphism $\Phi:M_0\times (0\cdot M\setminus \{\abf\} )\sqcup\{\abf\}\to M$ which is surjective if and only if $M$ is flasque. By \cite[Corollary 4.22]{Dias_Dinis(23)}, it follows that $M'/\ker^R(\Phi)$ is isomorphic to $M$.

\begin{proposition}\label{P:idempotent}
    Let $R$ be a unital commutative ring and $S$ an idempotent commutative monoid with a zero $\abf$. Then $M=R\times (S\setminus\{\abf\})\sqcup\{\abf\}$ with the operations:
    $$
    \begin{cases}
        (x,a)+(y,b)=(x+y,a\cdot b)\\
        (x,a)\cdot(y,b)=(x\cdot y,a\cdot b)\\
        (x,a) + \abf =\abf + (x,a)=(x,a) \cdot \abf =\abf \cdot (x,a)=\abf
    \end{cases}
    $$
    defines a common meadow.
\end{proposition}

 The proof of Proposition \ref{P:idempotent} is similar to the proof of Theorem \ref{T:Decomposition} and so it will be omitted.

Recall that an idempotent commutative monoid $S$ with zero (i.e. with an absorbent element)  is equivalent to a lattice, where  $x\leq y$ if $x\cdot y=y$ for  $x,y \in S$. The directed lattice associated with the meadow $M=R\times (S\setminus\{\abf\})\sqcup\{\abf\}$ is the lattice $S$ where each node is replaced by the ring $R$ and all transition maps are the identity.

\begin{exmp}
    Consider the monoid $S=\{0,x,y,\abf\}$, where $0$ is the identity, and $x\cdot y =\abf$. The common meadow $M=(\z_2\times \z_2)\times (S\setminus\{\abf\})\sqcup\{\abf\}$ is associated with the following directed lattice
    
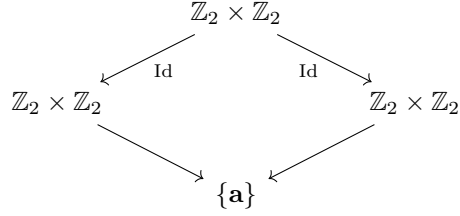
\[\begin{tikzcd}
	& {\z_2\times \z_2} \\
	{\z_2\times \z_2} && {\z_2\times \z_2} \\
	&\{\abf\}
	\arrow["\Id", from=1-2, to=2-1]
	\arrow["\Id"', from=1-2, to=2-3]
	\arrow[from=2-1, to=3-2]
	\arrow[from=2-3, to=3-2]
\end{tikzcd}\]
    Given an ideal $I$ of $M$, the quotient meadow $M/I$ (see \cite[Section~4]{Dias_Dinis(23)} for more on quotients of meadows) is a flasque pre-meadow with $\abf$ which is not necessarily a common meadow. Indeed, taking for instance the ideal  $I=\{((0,0),0),((0,1),x),((0,1),y),\abf\}$, the quotient $M/I$ is not a common meadow, while the quotient $M/J$ by the ideal $J=\{((0,0),0),((0,1),x),((1,0),y),\abf\}$ is in fact a common meadow.
\end{exmp}

Another consequence of the proof of Theorem \ref{T:Decomposition} is that if $S$ is not idempotent then $M=R\times (S\setminus\{\abf\})\sqcup\{\abf\}$  is a structure similar to a common meadow. However this structure is not distributive. In fact we have 
$$ (x,a)\cdot (y,b)+(x,a)\cdot (z,c)= (x,a)\cdot ( (y,b)+(z,c)) + (0,a),$$

and so distributivity holds up to a factor $(0,a)$. The fact that distributivity holds up to a correcting term resembles what happens with the so called \emph{neutrices} and in Nonstandard Analysis (see \cite{KoudjetiBerg(95),DinisBerg(book)}). Neutrices are convex additive subgroups of the hyperreal numbers, for example the set of all infinitesimals. Neutrices have also been studied from a purely algebraic point of view, giving rise to the notions of \emph{assembly}, \emph{association} and \emph{solid}, which are algebraic structures\footnote{These structures are related with usual algebraic structures, but are flexible in the sense that by having individualized neutral elements become stable under some (but not all) translations and shifts. For this reason, some authors have dubbed them ``flexible'' (see e.g.\ \cite{JB(12)}). More recently, assemblies have been shown to be unions of groups and can therefore be seen as semigroups \cite{DDT(24)}. This connection was also the inspiration to establish common meadows as unions of rings.} that possess individualized neutral elements. These can also be seen as ``generalized zeros'' and appear in the distributivity formula as correction terms, since distributivity does not hold in general (see \cite{DinisBerg(11),DinisBerg(ta)}). The paper  \cite{Dinis_Bottazzi} is a first attempt at connecting the two ideas, but a proper connection is still to be done.

By Theorem \ref{T:DirectedLattice}, a pre-meadow $P$ is a common meadow if and only if for all $x\in P$ the set $J_x=\{0\cdot z\in 0\cdot P\mid x+0\cdot z \in P_{0\cdot z}^{\times}\}$ has a greatest element. That is, the study of the elements of the set $J_x$ allows to check if a pre-meadow with $\abf$ is a common meadow. In general, this task is somewhat difficult. The next theorem gives a characterization, for flasque common meadows, which is easier to verify in practice. Indeed, it shows that, in that case, one only needs to verify the condition for the elements in $P_0$, which should be intuitively obvious since for flasque meadows $P_0$ determines the whole meadow.

\begin{theorem}\label{T:flasqueequiva}
    Let $P$ be a  flasque pre-meadow with $\abf$. Then $J_x$ has a greatest element, for all $x\in P$, if and only if it has a greatest element  for all $x\in P_0$. 
\end{theorem}
\begin{proof}
    The direct implication is obvious. Let $x\in P$. If $0\cdot x \in J_x$, then $J_x$ has a unique maximal element. So, let us assume that $0\cdot x\notin J_x$. Since $P$ is flasque,  there exists some $y\in P_0$ such that $x=y+0\cdot x$. Note that if $J_{x}\subseteq J_{y}$, then $0\cdot \overline{y}\cdot x\in J_x$, where  $\overline{y}$ is the greatest element of $J_y$. By the definition of $J_x$, for every element $0\cdot z\in J_x$ we have $0\cdot z \leq 0\cdot x$. Since $J_x\subseteq J_y$ we also have $0\cdot z \leq 0\cdot \overline{y}$. Hence $0\cdot z \leq 0\cdot \overline{y}\cdot x$, which implies that $0\cdot \overline{y}\cdot x$ is the greatest element of $J_x$.
\end{proof}
\begin{cexmp}
    In case a pre-meadow with $\abf$ is not flasque one is indeed  required to verify the condition of $J_x$ having a greatest element, for all $x$. To see this, consider the pre-meadow with $\abf$ defined by the directed lattice
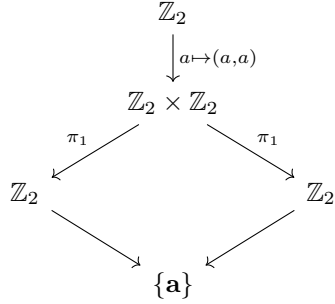
\[\begin{tikzcd}
	& {\z_2} \\
	& {\z_2\times \z_2} \\
	{\z_2} && {\z_2} \\
	& {\{\abf\}}
	\arrow["{a\mapsto (a,a)}", from=1-2, to=2-2]
	\arrow["{\pi_1}"', from=2-2, to=3-1]
	\arrow["{\pi_1}", from=2-2, to=3-3]
	\arrow[from=3-1, to=4-2]
	\arrow[from=3-3, to=4-2]
\end{tikzcd}\]
   The elements $x \in P_0\simeq \z_2$ all verify that $J_x$ has a greatest element. However, $J_{(1,0)}$ has two distinct maximal elements. So, the condition fails for an element of $P_{(0,0)}\simeq \z_2\times\z_2$.
\end{cexmp}

    As an application of Theorem \ref{T:flasqueequiva} we will see that there exist non-Artinian rings $R$ such that $M(R)$ is a common meadow. This provides a positive answer to a question raised in \cite{Dias_Dinis(Art)}.

    \begin{exmp}
        Consider the ring $R=\prod_{n\in\n}\z_2$, consisting of countably many copies of $\z_2$. Clearly $R$ is not Artinian. Let us see that $M(R)$ is a common meadow. Since $M(R)$ is flasque (by Proposition \ref{P:M(R)isflasque}), then by Theorem \ref{T:flasqueequiva} all we need to check that for all $x\in R=M(R)_0$ the set $J_x$ has a greatest element. To see that, consider $(a_n)_{n\in \n}\in R$, and define the element $(b_n)_{n\in\n}\in R$ by
        $$
        b_n=\begin{cases}
            b_n=0,& a_n\neq 0\\
            b_n=1,& a_n= 0.
        \end{cases} 
        $$
        Then $(a_n)_{n\in \n} + I$ is invertible in $R/I$ if and only if $(b_n)_{n\in\n}\in I$, and so the ideal generated by $(b_n)_{n\in\n}$ is the smallest ideal $I$ such that $(a_n)_{n\in \n} + I$ is invertible in $R/I$. Then, $I=((b_n)_{n\in\n})$ is the greatest element in $J_{(a_n)_{n\in \n}}$ and so it is a common meadow by Theorem \ref{T:DirectedLattice}.
    \end{exmp}

As shown by Counterexample~\ref{CE:Notallareflasque}, not all pre-meadows with $\abf$ are flasque. However, given a pre-meadow with $\abf$ we can consider the submeadow $P_0+0\cdot P$, which is a flasque pre-meadow with $\abf$. In fact, it is the largest flasque pre-meadow with $\abf$. We then have 
\[\begin{tikzcd}
	{P_0\times 0\cdot P} & {P_0+0\cdot P} & P
	\arrow[two heads, from=1-1, to=1-2]
	\arrow[hook, from=1-2, to=1-3]
\end{tikzcd}\]

where the first map is surjective and the second one is injective. So, given a pre-meadow with $\abf$ we can  use the simpler pre-meadows $P_0\times 0\cdot P$ and $P_0+0\cdot P$ in order to study $P$. 

\begin{proposition}
    Let $P$ be a pre-meadow with $\abf$. Then $\overline{P}=P_0+0\cdot P$ is a flasque pre-meadow with $\abf$. Additionally, $\overline{P}$ is the largest flasque submeadow of $P$ with $\overline{P}_0=P_0$. In particular, if $P$ is a common meadow then $\overline{P}$ is also a common meadow.
\end{proposition}
\begin{proof}
    By construction, one easily checks that $\overline{P}$ is in fact the largest flasque pre-meadow with $\abf$. 

    Suppose now that $P$ is a common meadow. Since $\overline{P}$ is flasque, by Theorem \ref{T:flasqueequiva} we have that $\overline{P}$ is a common meadow if and only if for all $x\in P_0$ the set $J_x$ has a greatest element. Since $J_x\subseteq 0\cdot P=0\cdot \overline{P}$ and $P$ is a common meadow,  $J_x$ has a greatest element and so $\overline{P}$ is also a common meadow.
\end{proof}

\begin{exmp}
Let $M$ be the common meadow associated with the following directed lattice  of rings
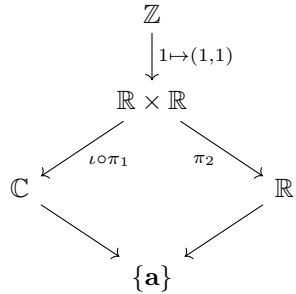
\[\begin{tikzcd}
	& \z \\
	& {\r \times \r} \\
	\cc && \r \\
	& {\{\abf\}}
	\arrow["{1\mapsto(1,1)}", from=1-2, to=2-2]
	\arrow["{\iota\circ\pi_1}", from=2-2, to=3-1]
	\arrow["{\pi_2}"', from=2-2, to=3-3]
	\arrow[from=3-1, to=4-2]
	\arrow[from=3-3, to=4-2]
\end{tikzcd}\]

Where $\iota$ is the inclusion map from $\r$ to $\cc$.
The meadow $M$ is not flasque since the map $1\mapsto (1,1)$ is not surjective. The construction of the  largest flasque pre-meadow is done by taking the image of the transition maps. Indeed, the directed lattice
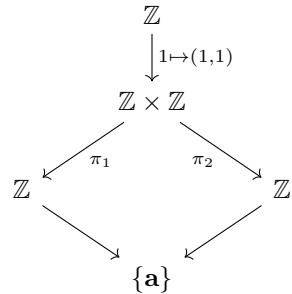
\[\begin{tikzcd}
	& \z \\
	& {\z\times\z} \\
	\z && \z \\
	& {\{\abf\}}
	\arrow["{1\mapsto(1,1)}", from=1-2, to=2-2]
	\arrow["{\pi_1}", from=2-2, to=3-1]
	\arrow["{\pi_2}"', from=2-2, to=3-3]
	\arrow[from=3-1, to=4-2]
	\arrow[from=3-3, to=4-2]
\end{tikzcd}\]
 is a flasque pre-meadow with $\abf$. In this case, it is even a common meadow.
\end{exmp}

\bibliographystyle{siam}
\bibliography{References}
\subsection*{Acknowledgments}
Both authors acknowledge the support of FCT - Funda\c{c}\~ao para a Ci\^{e}ncia e Tecnologia under the project: 10.54499/UIDB/04674/2020, and the research center CIMA -- Centro de Investigação em Matemática e Aplicações. 

The second author also acknowledges the support of CMAFcIO -- Centro de Matem\'{a}tica, Aplica\c{c}\~{o}es Fundamentais e Investiga\c{c}\~{a}o Operacional under the project UIDP/04561/2020.

\end{document}